\documentclass{llncs}
\usepackage{makeidx}  
\usepackage{amsmath, amsfonts, amssymb}
\usepackage{url}
\usepackage{graphicx,subfig}
\usepackage[misc,geometry]{ifsym}

\usepackage[ruled]{algorithm2e}

\SetAlFnt{\small}
\SetAlCapFnt{\small}
\SetAlCapNameFnt{\small}
\SetAlCapHSkip{0pt}
\IncMargin{-\parindent}

\def\R{\mathbb R}
\def\e{\varepsilon}
\def\vp{\varphi}

\newcommand{\la}{\langle}
\newcommand{\ra}{\rangle}

\graphicspath{{./pics2016-04-28/}}

\begin{document}
%
%
\pagestyle{headings}  
%
%
\title{Fast Primal-Dual Gradient Method for Strongly Convex Minimization Problems with Linear Constraints}
\titlerunning{Fast Primal-Dual GM}  
%
\author{
Alexey Chernov\inst{1} \and
Pavel Dvurechensky\inst{2} \and
Alexender Gasnikov\inst{3}
}
\authorrunning{Alexey Chernov et al.} 
%
\tocauthor{Alexey Chernov, Pavel Dvurechensky, Alexender Gasnikov}
\institute{Moscow Institute of Physics and Technology, Dolgoprudnyi 141700, Moscow Oblast, Russia,\\
\email{alexmipt@mail.ru}
\and
Weierstrass Institute for Applied Analysis and Stochastics, Berlin 10117, Germany, \\
Institute for Information Transmission Problems, Moscow 127051, Russia,
\email{pavel.dvurechensky@wias-berlin.de}
\and
Moscow Institute of Physics and Technology, Dolgoprudnyi 141700, Moscow Oblast, Russia,\\
Institute for Information Transmission Problems, Moscow 127051, Russia, 
\email{gasnikov@yandex.ru}
}

\maketitle              

\begin{abstract}
In this paper we consider a class of optimization problems with a strongly convex objective function and the feasible set given by an intersection of a simple convex set with a set given by a number of linear equality and inequality constraints. A number of optimization problems in applications can be stated in this form, examples being the entropy-linear programming, the ridge regression, the elastic net, the regularized optimal transport, etc.  
We extend the Fast Gradient Method applied to the dual problem in order to make it primal-dual so that it allows not only to solve the dual problem, but also to construct nearly optimal and nearly feasible solution of the primal problem. 
We also prove a theorem about the convergence rate for the proposed algorithm in terms of the objective function and the linear constraints infeasibility.
\end{abstract}
\section*{Introduction}
In this paper we consider a constrained convex optimization problem of the following form
\begin{equation}
(P_1) \quad \quad \min_{x\in Q \subseteq E} \left\{ f(x) : A_1x =b_1, A_2x \leq b_2\right\},
\notag
\end{equation}
where $E$ is a finite-dimensional real vector space, $Q $ is a simple closed convex set, $A_1$, $A_2$ are given linear operators from $E$ to some finite-dimensional real vector spaces $H_1$ and $H_2$ respectively, $b_1 \in H_1$, $b_2 \in H_2$ are given, $f(x)$ is a $\nu$-strongly convex function  on $Q$ with respect to some chosen norm $\|\cdot\|_E$ on $E$. The last means that for any $x,y \in Q$ $f(y) \geq f(x) + \la \nabla f(x) , y-x \ra + \frac{\nu}{2}\|x-y\|^2_E$, where $\nabla f(x)$ is any subgradient of $f(x)$ at $x$ and hence is an element of the dual space $E^*$. Also we denote the value of a linear function $g \in E^*$ at $x\in E$ by $\la g, x \ra$.

Problem $(P_1)$ captures a broad set of optimization problems arising in applications. The first example is the classical entropy-linear programming (ELP) problem \cite{Fang_1997_Entropy} which arises in many applications such as econometrics \cite{Golan_1996_Maximum}, modeling in science and engineering \cite{Kapur_1989_Maximum}, especially in the modeling of traffic flows \cite{Gasnikov_2013_Introduction} and the IP traffic matrix estimation \cite{Rahman_2006_IP,Zhang_2005_Estimating}. Other examples are the ridge regression problem \cite{Hastie_2009_Elements} and the elastic net approach \cite{Zou_2005_Regularization} which are used in machine learning. Finally, the problem class $(P_1)$ covers problems of regularized optimal transport (ROT) \cite{Cuturi_2013_Sinkhorn} and regularized optimal partial transport (ROPT) \cite{Benamou_2015_Iterative}, which recently have become popular in application to the image analysis.

The classical balancing algorithms such as \cite{Cuturi_2013_Sinkhorn,Bregman_1967_Relaxation,Bregman_1967_Proof} are very efficient for solving ROT problems or special types of ELP problem, but they can deal only with linear equality constraints of special type and their rate of convergence estimates are rather impractical \cite{Franklin_1989_Scaling}.  
In \cite{Benamou_2015_Iterative} the authors provide a generalization but only for the ROPT problems which are a particular case of Problem $(P_1)$ with linear inequalities constraints of a special type and no convergence rate estimates are provided.
Unfortunately the existing balancing-type algorithms for the ROT and ROPT problems become very unstable when the regularization parameter is chosen very small, which is the case when one needs to calculate a good approximation to the solution of the optimal transport (OT) or the optimal partial transport (OPT) problem.

In practice the typical dimensions of the spaces $E,H_1,H_2$ range from thousands to millions, which makes it natural to use a first-order method to solve Problem $(P_1)$. 
A common approach to solve such large-scale Problem $(P_1)$ is to make the transition to the Lagrange dual problem and solve it by some first-order method. Unfortunately, existing methods which elaborate this idea have at least two drawbacks. 
Firstly, the convergence analysis of the Fast Gradient Method (FGM) \cite{Nesterov_2005_Smooth} can not be directly applied since it is based on the assumption of boundedness of the feasible set in both the primal and the dual problem, which does not hold for the Lagrange dual problem.
A possible way to overcome this obstacle is to assume that the solution of the dual problem is bounded and add some additional constraints to the Lagrange dual problem in order to make the dual feasible set bounded. But in practice the bound for the solution of the dual problem is  usually not known. 
In \cite{Gasnikov_2015_Entropy} the authors use this approach with additional constraints and propose a restart technique to define the unknown bound for the optimal dual variable value. 
Unfortunately, the authors consider only the classical ELP problem with only the equality constraints and it is not clear whether their technique can be applied for Problem $(P_1)$ with inequality constraints.  
Secondly, it is important to estimate the rate of convergence not only in terms of the error in the solution of the Lagrange dual problem at it is done in \cite{Polyak_2013_Dual,Necoara_2008_Applications} but also in terms of the error in the solution of the primal problem \footnote{The absolute value here is crucial since $x_k$ may not satisfy linear constraints and hence $f(x_k)-Opt[P_1]$ could be negative.} $|f(x_k)-Opt[P_1]|$ and the linear constraints infeasibility $\|A_1x_k-b_1\|_{H_1}$, $\|(A_2x_k-b_2)_+\|_{H_2}$, where vector $v_+$ denotes the vector with components $[v_+]_i=(v_i)_+=\max\{v_i,0\}$, $x_k$ is the output of the algorithm on the $k$-th iteration, $Opt[P_1]$ denotes the optimal function value for Problem $(P_1)$.
Alternative approaches \cite{Goldstein_2012_Fast,Shefi_2014_Rate} based on the idea of the method of multipliers and the quasi-Newton methods such as L-BFGS also do not allow to obtain the convergence rate for the approximate primal solution and the linear constraints infeasibility. 

Our contributions in this work are the following. 
We extend the Fast Gradient Method \cite{Nesterov_2005_Smooth,Devolder_2014_First} applied to the dual problem in order to make it primal-dual so that it allows not only to solve the dual problem, but also to construct nearly optimal and nearly feasible solution to the primal problem $(P_1)$. We also equip our method with a stopping criterion which allows an online control of the quality of the approximate primal-dual solution.
Unlike \cite{Cuturi_2013_Sinkhorn,Benamou_2015_Iterative,Polyak_2013_Dual,Necoara_2008_Applications,Goldstein_2012_Fast,Shefi_2014_Rate,Gasnikov_2015_Entropy} we provide the estimates for the rate of convergence in terms of the error in the solution of the primal problem $|f(x_k)-Opt[P_1]|$ and the linear constraints infeasibility $\|A_1x_k-b_1\|_{H_1}$, $\|(A_2x_k-b_2)_+\|_{H_2}$.
In the contrast to the estimates in \cite{Nesterov_2005_Smooth}, our estimates do not rely on the assumption that the feasible set of the dual problem is bounded. 
At the same time our approach is applicable for the wider class of problems defined by $(P_1)$ than approaches in \cite{Cuturi_2013_Sinkhorn,Gasnikov_2015_Entropy}. 
In the computational experiments we show that our approach allows to solve 
ROT problems more efficiently than the algorithms of \cite{Cuturi_2013_Sinkhorn,Benamou_2015_Iterative,Gasnikov_2015_Entropy} when the regularization parameter is small.


\section{Preliminaries}
\label{S:prel}

\subsection{Notation}
For any finite-dimensional real vector space $E$ we denote by $E^*$ its dual. We denote the value of a linear function $g \in E^*$ at $x\in E$ by $\la g, x \ra$. Let $\|\cdot\|_E$ denote some norm on $E$ and $\|\cdot\|_{E,*}$ denote the norm on $E^*$ which is dual to $\|\cdot\|_E$
$$
\|g\|_{E,*} = \max_{\|x\|_E \leq 1} \la g, x \ra.
$$
In the special case when $E$ is a Euclidean space we denote the standard Euclidean norm by $\|\cdot\|_2$. Note that in this case the dual norm is also Euclidean.
By $\partial f(x)$ we denote the subdifferential of the function $f(x)$ at a point $x$. Let $E_1, E_2$ be two finite-dimensional real vector spaces. For a linear operator $A:E_1 \to E_2$ we define its norm as follows
$$
\|A\|_{E_1 \to E_2} = \max_{x \in E_1,u \in E_2^*} \{\la u, A x \ra : \|x\|_{E_1} = 1, \|u\|_{E_2,*} = 1 \}.
$$
For a linear operator $A:E_1 \to E_2$ we define the adjoint operator $A^T: E_2^* \to E_1^*$ in the following way
$$
\la u, A x \ra = \la A^T u, x \ra, \quad \forall u \in E_2^*, \quad x \in E_1.
$$
We say that a function $f: E \to \R$ has a $L$-Lipschitz-continuous gradient if it is differentiable and its gradient satisfies Lipschitz condition
$$
\|\nabla f(x) - \nabla f(y) \|_{E,*} \leq L \|x-y\|_E.
$$


We characterize the quality of an approximate solution to Problem $(P_1)$ by three quantities $\e_f,\e_{eq}, \e_{in} > 0$ and say that a point $\hat{x}$ is an $(\e_f,\e_{eq}, \e_{in})$-solution to Problem $(P_1)$ if the following inequalities hold 
\begin{equation}
|f(\hat{x}) - Opt[P_1]| \leq \e_f , \quad \|A_1\hat{x}-b_1\|_2 \leq \e_{eq}, \quad \|(A_2\hat{x}-b_2)_+\|_2 \leq \e_{in}.
\label{eq:sol_def}
\end{equation}
Here $Opt[P_1]$ denotes the optimal function value for Problem $(P_1)$ and the vector $v_+$ denotes the vector with components $[v_+]_i=(v_i)_+=\max\{v_i,0\}$.
Also for any $t \in R$ we denote by  $\lceil t \rceil$ the smallest integer greater than or equal to $t$.

\subsection{Dual Problem}

Let us denote 
$\Lambda =\{\lambda = (\lambda^{(1)},\lambda^{(2)})^T \in H_1^* \times H_2^*: \lambda^{(2)} \geq 0\}$. The Lagrange dual problem to Problem $(P_1)$ is
\begin{equation}
(D_1) \quad \quad \max_{\lambda \in \Lambda} \left\{ - \la \lambda^{(1)}, b_1 \ra - \la \lambda^{(2)}, b_2 \ra + \min_{x\in Q} \left( f(x) + \la A_1^T \lambda^{(1)} + A_2^T \lambda^{(2)} ,x \ra \right) \right\}.
\notag
\end{equation}


We rewrite Problem $(D_1)$ in the equivalent form of a minimization problem.
\begin{align}
& (P_2) \quad \min_{\lambda \in \Lambda} \left\{   \la \lambda^{(1)}, b_1 \ra + \la \lambda^{(2)}, b_2 \ra + \max_{x\in Q} \left( -f(x) - \la A_1^T \lambda^{(1)} + A_2^T \lambda^{(2)} ,x \ra \right) \right\}. \notag
\end{align}
We denote
\begin{equation}
\vp(\lambda) = \vp(\lambda^{(1)}, \lambda^{(2)}) = \la \lambda^{(1)}, b_1 \ra + \la \lambda^{(2)}, b_2 \ra + \max_{x\in Q} \left( -f(x) - \la A_1^T \lambda^{(1)} + A_2^T \lambda^{(2)} ,x \ra \right)
\label{eq:vp_def}
\end{equation}
Note that the gradient of the function $\vp(\lambda)$ is equal to (see e.g. \cite{Nesterov_2005_Smooth})
\begin{equation}
\nabla \vp(\lambda) = \left(
\begin{aligned}
b_1 - A_1 x (\lambda)\\
b_2 - A_2 x (\lambda)\\
\end{aligned}
 \right),
\label{eq:nvp}
\end{equation}
where $x (\lambda)$ is the unique solution of the problem
\begin{equation}
\max_{x\in Q} \left( -f(x) - \la A_1^T \lambda^{(1)} + A_2^T \lambda^{(2)} ,x \ra \right).
\label{eq:inner}
\end{equation}
Note that this gradient is Lipschitz-continuous (see e.g. \cite{Nesterov_2005_Smooth}) with constant
$$
L = \frac{1}{\nu}\left(\|A_1\|_{E \to H_1}^2+ \|A_2\|_{E \to H_2}^2\right).
$$
It is obvious that
\begin{equation}
Opt[D_1]=-Opt[P_2].
\label{eq:D_1_P_2_sol}
\end{equation}
Here by $Opt[D_1]$, $Opt[P_2]$ we denote the optimal function value in Problem $(D_1)$ and Problem $(P_2)$ respectively. 
Finally, the following inequality follows from the weak duality
\begin{equation}
Opt[P_1] \geq Opt[D_1].
\label{eq:wD}
\end{equation}


\subsection{Main Assumptions}
\label{S:main_assum}
We make the following two main assumptions
\begin{enumerate}
	\item The problem \eqref{eq:inner} is simple in the sense that for any $x \in Q$ it has a closed form solution or can be solved very fast up to the machine precision.
	\item The dual problem $(D_1)$ has a solution $\lambda^*=(\lambda^{*(1)},\lambda^{*(2)})^T$ and there exist some $R_1, R_2 >0$ such that
	\begin{equation}
	\|\lambda^{*(1)}\|_{2} \leq R_1 < +\infty, \quad \|\lambda^{*(2)}\|_{2} \leq R_2 < +\infty. 
	\label{eq:l_bound}
	\end{equation}	 
\end{enumerate}

\subsection{Examples of Problem $(P_1)$}
In this subsection we describe several particular problems which can be written in the form of Problem $(P_1)$.

\textbf{Entropy-linear programming problem \cite{Fang_1997_Entropy}.}

\begin{equation}
\min_{x \in S_n(1)} \left\{ \sum_{i=1}^n x_i \ln \left(x_i/\xi_i\right)  : Ax =b \right\}
\notag
\end{equation}
for some given $\xi \in \R^n_{++} = \{x \in \R^n: x_i > 0 , i=1,...,n \}$. Here $S_n(1) = \{x \in \R^n : \sum_{i=1}^n x_i = 1, x_i \geq 0 , i=1,...,n \}$.

\textbf{Regularized optimal transport problem \cite{Cuturi_2013_Sinkhorn}.}
\begin{equation}
\min_{X \in \R^{p \times p}_+} \left\{ \gamma \sum_{i,j=1}^p x_{ij} \ln x_{ij} + \sum_{i,j=1}^p c_{ij} x_{ij}: Xe=a_1, X^Te=a_2 \right\},
\label{eq:ROT}
\end{equation}
where $e \in \R^p$ is the vector of all ones, $a_1,a_2 \in S_p(1)$, $c_{ij} \geq 0, i,j=1,...,p$ are given, $\gamma > 0$ is the regularization parameter, $X^T$ is the transpose matrix of $X$, $x_{ij}$ is the element of the matrix $X$ in the $i$th row and the $j$th column.

\textbf{Regularized optimal partial transport problem \cite{Benamou_2015_Iterative}.}
\begin{equation}
\min_{X \in \R^{p \times p}_+} \left\{ \gamma \sum_{i,j=1}^p x_{ij} \ln x_{ij} + \sum_{i,j=1}^p c_{ij} x_{ij}: Xe \leq a_1, X^Te \leq a_2, e^T X e = m \right\},
\notag
\end{equation}
where $a_1,a_2 \in \R^p_+$, $c_{ij} \geq 0, i,j=1,...,p$, $m >0$ are given, $\gamma > 0$ is the regularization parameter  and the inequalities should be understood component-wise.

\section{Algorithm and Theoretical Analysis}
\label{S:2}
We extend the Fast Gradient Method \cite{Nesterov_2005_Smooth,Devolder_2014_First} in order to make it primal-dual so that it allows not only to solve the dual problem $(P_2)$, but also to construct a nearly optimal and nearly feasible solution to the primal problem $(P_1)$. We also equip it with a stopping criterion which allows an online control of the quality of the approximate primal-dual solution.
Let $\{\alpha_i\}_{i \geq 0}$ be a sequence of coefficients satisfying
\begin{align}
& \alpha_0 \in (0,1],  \notag \\ 
& \alpha_k^2  \leq \sum_{i=0}^k{\alpha_i} , \quad \forall k \geq 1. \notag
\end{align}
We define also $C_k = \sum_{i=0}^k{\alpha_i}$ and $\tau_{i}=\frac{\alpha_{i+1}}{C_{i+1}}$. Usual choice is $\alpha_i=\frac{i+1}{2}$, $i\geq 0$. In this case $C_k=\frac{(k+1)(k+2)}{4}$. Also we define the Euclidean norm on $H_1^* \times H_2^*$ in a natural way
$$
\|\lambda\|_2^2=\|\lambda^{(1)}\|_2^2 + \|\lambda^{(2)}\|_2^2
$$
for any $\lambda=(\lambda^{(1)},\lambda^{(2)})^T \in H_1^* \times H_2^*$.
Unfortunately we can not directly use the convergence results of \cite{Nesterov_2005_Smooth,Devolder_2014_First} for the reason that the feasible set $\Lambda$ in the dual problem $(D_1)$ is unbounded and the constructed sequence $\hat{x}_k$ may possibly not satisfy the equality and inequality constraints. 

\begin{algorithm}[h!]
\SetAlgoNoLine
\KwIn{The sequence $\{\alpha_i\}_{i \geq 0}$, accuracy $\tilde{\e}_f,\tilde{\e}_{eq},\tilde{\e}_{in} > 0$}
\KwOut{The point $\hat{x}_k$.}

Choose  $\lambda_0 = (\lambda^{(1)}_0,\lambda^{(2)}_0)^T = 0$.

Set $k$ = 0.

\Repeat{$|f(\hat{x}_k)+\vp(\eta_k)| \leq \tilde{\e}_f$, $\|A_1\hat{x}_k-b_1\|_{2} \leq \tilde{\e}_{eq}$, $\|(A_2\hat{x}_k-b_2)_+\|_{2} \leq \tilde{\e}_{in}$}{
Compute
\begin{equation}
\eta_k = (\eta_k^{(1)},\eta_k^{(2)})^T = \arg \min_{\lambda \in \Lambda} \left\{ \vp(\lambda_k) + \la \nabla \vp(\lambda_k), \lambda-\lambda_k \ra + \frac{L}{2} \|\lambda-\lambda_k\|_2^2 \right\}.
\notag
\end{equation}

\begin{equation}
\zeta_k = (\zeta_k^{(1)},\zeta_k^{(2)})^T  = \arg \min_{\lambda \in \Lambda} \left\{  \sum_{i=0}^k{\alpha_i \left( \vp(\lambda_i) + \la \nabla \vp(\lambda_i), \lambda-\lambda_i \ra \right) } + \frac{L}{2} \|\lambda\|_2^2 \right\}.
\notag
\end{equation}
Set
\begin{equation}
\lambda_{k+1} = (\lambda_{k+1}^{(1)},\lambda_{k+1}^{(2)})^T  = \tau_k \zeta_k + (1- \tau_k) \eta_k.
\notag
\end{equation}
Set
\begin{equation}
\hat{x}_k = \frac{1}{C_k}\sum_{i=0}^k \alpha_i x(\lambda_i) = (1-\tau_{k-1})\hat{x}_{k-1}+\tau_{k-1}x(\lambda_k).
\notag
\end{equation}
Set $k=k+1$.

    }
\caption{Fast Primal-Dual Gradient Method}
\label{alg:FGM_gen}
\end{algorithm}



\begin{theorem}
Let the assumptions listed in the subsection \ref{S:main_assum} hold and $\alpha_i=\frac{i+1}{2}$, $i\geq 0$ in Algorithm \ref{alg:FGM_gen}. Then Algorithm \ref{alg:FGM_gen} will stop after not more than
$$
N_{stop} = \max \left\{ \left\lceil \sqrt{\frac{8L(R_1^2+R_2^2)}{\tilde{\e}_f}}  \right\rceil, \left\lceil \sqrt{\frac{8L(R_1^2+R_2^2)}{R_1\tilde{\e}_{eq}}} \right\rceil, \left\lceil \sqrt{\frac{8L(R_1^2+R_2^2)}{R_2\tilde{\e}_{in}}} \right\rceil \right\} -1
$$
iterations. 
Moreover after not more than 
$$
N = \max \left\{ \left\lceil \sqrt{\frac{16L(R_1^2+R_2^2)}{\e_f}}  \right\rceil, \left\lceil \sqrt{\frac{8L(R_1^2+R_2^2)}{R_1\e_{eq}}} \right\rceil, \left\lceil \sqrt{\frac{8L(R_1^2+R_2^2)}{R_2\e_{in}}} \right\rceil \right\} -1
$$
iterations of Algorithm \ref{alg:FGM_gen} the point $\hat{x}_N$ will be an approximate solution to Problem $(P_1)$ in the sense of \eqref{eq:sol_def}.
\label{Th:FGM_compl}
\end{theorem}
\begin{proof}

From the complexity analysis of the FGM \cite{Nesterov_2005_Smooth,Devolder_2014_First} one has
\begin{equation}
C_k \vp(\eta_k) \leq \min_{\lambda \in \Lambda} \left\{  \sum_{i=0}^k{\alpha_i \left( \vp(\lambda_i) + \la \nabla \vp(\lambda_i), \lambda-\lambda_i \ra \right) } + \frac{L}{2} \|\lambda\|_2^2 \right\}
\label{eq:FGM_compl}
\end{equation}
Let us introduce a set $\Lambda_R =\{\lambda = (\lambda^{(1)},\lambda^{(2)})^T: \lambda^{(2)} \geq 0, \|\lambda^{(1)}\|_2 \leq 2R_1, \|\lambda^{(2)}\|_2 \leq 2R_2 \}$ where $R_1$, $R_2$ are given in \eqref{eq:l_bound}. Then from \eqref{eq:FGM_compl} we obtain
\begin{align}
& C_k \vp(\eta_k) \leq \min_{\lambda \in \Lambda} \left\{  \sum_{i=0}^k{\alpha_i \left( \vp(\lambda_i) + \la \nabla \vp(\lambda_i), \lambda-\lambda_i \ra \right) } + \frac{L}{2} \|\lambda\|_2^2 \right\} \leq \notag \\
& \leq \min_{\lambda \in \Lambda_R} \left\{  \sum_{i=0}^k{\alpha_i \left( \vp(\lambda_i) + \la \nabla \vp(\lambda_i), \lambda-\lambda_i \ra \right) } + \frac{L}{2} \|\lambda\|_2^2 \right\} \leq \notag \\
& \leq \min_{\lambda \in \Lambda_R} \left\{  \sum_{i=0}^k{\alpha_i \left( \vp(\lambda_i) + \la \nabla \vp(\lambda_i), \lambda-\lambda_i \ra \right) } \right\} + 2L(R_1^2+R_2^2).
\label{eq:proof_st_1}
\end{align}
On the other hand from the definition \eqref{eq:vp_def} of $\vp(\lambda)$ we have
\begin{align}
& \vp(\lambda_i) = \vp(\lambda_i^{(1)},\lambda_i^{(2)}) = \la \lambda_i^{(1)}, b_1 \ra + \la \lambda_i^{(2)}, b_2 \ra + \notag \\
& +  \max_{x\in Q} \left( -f(x) - \la A_1^T \lambda_i^{(1)} + A_2^T \lambda_i^{(2)} ,x \ra \right) = \notag \\
& = \la \lambda_i^{(1)}, b_1 \ra + \la \lambda_i^{(2)}, b_2 \ra - f(x(\lambda_i)) - \la A_1^T \lambda_i^{(1)} + A_2^T \lambda_i^{(2)} ,x(\lambda_i) \ra . \notag
\end{align}
Combining this equality with \eqref{eq:nvp} we obtain
\begin{align}
& \vp(\lambda_i) - \la \nabla \vp (\lambda_i), \lambda_i \ra = \vp(\lambda_i^{(1)},\lambda_i^{(2)}) - \la \nabla \vp(\lambda_i^{(1)},\lambda_i^{(2)}), (\lambda_i^{(1)},\lambda_i^{(2)})^T \ra = \notag \\
& = \la \lambda_i^{(1)}, b_1 \ra + \la \lambda_i^{(2)}, b_2 \ra - f(x(\lambda_i)) - \la A_1^T \lambda_i^{(1)} + A_2^T \lambda_i^{(2)} ,x(\lambda_i) \ra -\notag \\
& - \la b_1-A_1 x(\lambda_i),\lambda_i^{(1)} \ra - \la b_2-A_2 x(\lambda_i),\lambda_i^{(2)} \ra = - f(x(\lambda_i)). \notag
\end{align}
Summing these inequalities from $i=0$ to $i=k$ with the weights $\{\alpha_i\}_{i=1,...k}$ we get, using the convexity of $f(\cdot)$
\begin{align}
&  \sum_{i=0}^k{\alpha_i \left( \vp(\lambda_i) + \la \nabla \vp(\lambda_i), \lambda-\lambda_i \ra \right) } = \notag \\
& = -\sum_{i=0}^k \alpha_i f(x(\lambda_i)) + \sum_{i=0}^k \alpha_i \la (b_1-A_1 x(\lambda_i),b_2-A_2 x(\lambda_i))^T, (\lambda^{(1)},\lambda^{(2)})^T \ra \leq \notag \\
& \leq -C_kf(\hat{x}_k) + C_k \la (b_1-A_1 \hat{x}_k,b_2-A_2 \hat{x}_k)^T, (\lambda^{(1)},\lambda^{(2)})^T \ra . \notag
\end{align}
Substituting this inequality to \eqref{eq:proof_st_1} we obtain
\begin{align}
& C_k \vp(\eta_k)  \leq -C_kf(\hat{x}_k) + \notag \\
& + C_k \min_{\lambda \in \Lambda_R} \left\{  \la (b_1-A_1 \hat{x}_k,b_2-A_2 \hat{x}_k)^T, (\lambda^{(1)},\lambda^{(2)})^T \ra \right\} + 2L(R_1^2+R_2^2) . \notag
\end{align}
Finally, since 
\begin{align}
&\max_{\lambda \in \Lambda_R} \left\{  \la (-b_1+A_1 \hat{x}_k,-b_2+A_2 \hat{x}_k)^T, (\lambda^{(1)},\lambda^{(2)})^T \ra \right\} = \notag \\
& =2 R_1 \|A_1 \hat{x}_k - b_1 \|_2 + 2R_2 \| (A_2 \hat{x}_k - b_2)_+ \|_2, \notag
\end{align} 
we obtain
\begin{equation}
\vp(\eta_k) + f(\hat{x}_k) +2 R_1 \|A_1 \hat{x}_k - b_1 \|_2 + 2R_2 \| (A_2 \hat{x}_k - b_2)_+ \|_2   \leq \frac{2L(R_1^2+R_2^2)}{C_k}.
\label{eq:vpmfxh}
\end{equation}
Since  $\lambda^*=(\lambda^{*(1)}, \lambda^{*(2)})^T$ is an optimal solution of Problem $(D_1)$ we have for any $x \in Q$
$$
Opt[P_1]\leq f(x) + \la  \lambda^{*(1)}, A_1 x -b_1 \ra + \la \lambda^{*(2)}, A_2 x -b_2 \ra.
$$
Using the assumption \eqref{eq:l_bound} and that $\lambda^{*(2)} \geq 0$ we get
\begin{equation}
f(\hat{x}_k) \geq Opt[P_1]- R_1 \|A_1 \hat{x}_k - b_1 \|_2- R_2 \|(A_2 \hat{x}_k - b_2)_+ \|_2.
\label{eq:fxhat_est}
\end{equation}
Hence
\begin{align}
& \vp(\eta_k) + f(\hat{x}_k)  = \vp(\eta_k) - Opt[P_2]+Opt[P_2] + Opt[P_1]  - Opt[P_1] + f(\hat{x}_k) \stackrel{\eqref{eq:D_1_P_2_sol}}{=} \notag \\
& = \vp(\eta_k)  - Opt[P_2]-Opt[D_1]+Opt[P_1]  - Opt[P_1] + f(\hat{x}_k)  
 \stackrel{\eqref{eq:wD}}{\geq} \notag \\
&  \geq - Opt[P_1] + f(\hat{x}_k) \stackrel{\eqref{eq:fxhat_est}}{\geq} - R_1 \|A_1 \hat{x}_k - b_1 \|_2- R_2 \|(A_2 \hat{x}_k - b_2)_+ \|_2.
\label{eq:aux1}
\end{align}
This and \eqref{eq:vpmfxh} give
\begin{equation}
R_1 \|A_1 \hat{x}_k - b_1 \|_2 + R_2 \| (A_2 \hat{x}_k - b_2)_+ \|_2 \leq \frac{2L(R_1^2+R_2^2)}{C_k}.
\label{eq:R_norm_est}
\end{equation}
Hence we obtain
\begin{equation}
\vp(\eta_k) + f(\hat{x}_k) \stackrel{\eqref{eq:aux1},\eqref{eq:R_norm_est}}{\geq} - \frac{2L(R_1^2+R_2^2)}{C_k}.
\label{eq:vppfxhatgeq}
\end{equation}
On the other hand we have 
\begin{equation}
\vp(\eta_k) + f(\hat{x}_k) \stackrel{\eqref{eq:vpmfxh}}{\leq} \frac{2L(R_1^2+R_2^2)}{C_k}.
\label{eq:vppfxhatleq}
\end{equation}
Combining \eqref{eq:R_norm_est}, \eqref{eq:vppfxhatgeq}, \eqref{eq:vppfxhatleq} we conclude
\begin{align}
&\|A_1 \hat{x}_k - b_1 \|_2 \leq \frac{2L(R_1^2+R_2^2)}{C_kR_1}, \notag \\
&\| (A_2 \hat{x}_k - b_2)_+ \|_2 \leq \frac{2L(R_1^2+R_2^2)}{C_kR_2 }, \notag \\
&|\vp(\eta_k) + f(\hat{x}_k)| \leq \frac{2L(R_1^2+R_2^2)}{C_k}.
\label{eq:untileq}
\end{align}
As we know for the chosen sequence $\alpha_i = \frac{i+1}{2}, i \geq 0$ it holds that $C_k=\frac{(k+1)(k+2)}{4} \geq \frac{(k+1)^2}{4}$.
Then in accordance to \eqref{eq:untileq} after given in the theorem statement number $N_{stop}$ of the iterations of Algorithm \ref{alg:FGM_gen} the stopping criterion will fulfill and Algorithm \ref{alg:FGM_gen} will stop.

Now let us prove the second statement of the theorem. 
We have 
\begin{align}
& \vp(\eta_k) + Opt[P_1] = \vp(\eta_k)  - Opt[P_2]+Opt[P_2]+Opt[P_1] \stackrel{\eqref{eq:D_1_P_2_sol}}{=} \notag \\
& = \vp(\eta_k)  - Opt[P_2]-Opt[D_1]+Opt[P_1] \stackrel{\eqref{eq:wD}}{\geq} 0. 
\notag
\end{align} 
Hence
\begin{equation}
f(\hat{x}_k)- Opt[P_1] \leq f(\hat{x}_k) + \vp(\eta_k)
\label{eq:fxhatmopt}
\end{equation}
On the other hand 
\begin{equation}
f(\hat{x}_k) - Opt[P_1] \stackrel{\eqref{eq:fxhat_est}}{\geq}  - R_1 \|A_1 \hat{x}_k - b_1 \|_2- R_2 \|(A_2 \hat{x}_k - b_2)_+ \|_2.
\label{eq:fxhat_est1}
\end{equation}
Note that since the point $\hat{x}_k$ may not satisfy the equality and inequality constraints one can not guarantee that $f(\hat{x}_k) - Opt[P_1] \geq 0$.
From \eqref{eq:fxhatmopt}, \eqref{eq:fxhat_est1} we can see that if we set $\tilde{\e}_f = \e_f$, $\tilde{\e}_{eq} = \min\{\frac{\e_f}{2R_1},\e_{eq}\}$, $\tilde{\e}_{in} = \min\{\frac{\e_f}{2R_2},\e_{in}\}$ and run Algorithm \ref{alg:FGM_gen} for $N$ iterations, where $N$ is given in the theorem statement, we obtain that \eqref{eq:sol_def} fulfills and $\hat{x}_N$ is an approximate solution to Problem $(P_1)$ in the sense of \eqref{eq:sol_def}. \qed
\end{proof}
Note that other authors \cite{Cuturi_2013_Sinkhorn,Benamou_2015_Iterative,Polyak_2013_Dual,Necoara_2008_Applications,Goldstein_2012_Fast,Shefi_2014_Rate,Gasnikov_2015_Entropy} do not provide the complexity analysis for their algorithms when the accuracy of the solution is given by \eqref{eq:sol_def}.

\section{Preliminary Numerical Experiments}
\label{S:3}
To compare our algorithm with the existing algorithms we choose the problem \eqref{eq:ROT} of regularized optimal transport \cite{Cuturi_2013_Sinkhorn} which is a special case of Problem $(P_1)$. The first reason is that despite insufficient theoretical analysis the existing balancing type methods for solving this type of problems are known to be very efficient in practice \cite{Cuturi_2013_Sinkhorn} and provide a kind of benchmark for any new method. The second reason is that ROT problem have recently become very popular in application to image analysis based on Wasserstein spaces geometry \cite{Cuturi_2013_Sinkhorn,Benamou_2015_Iterative}.

Our numerical experiments were carried out on a PC with CPU Intel Core i5 (2.5Hgz), 2Gb of RAM using Matlab 2012 (8.0).
We compare proposed in this article Algorithm \ref{alg:FGM_gen} (below we refer to it as FGM) with the following algorithms
\begin{itemize}
	\item Applied to the dual problem $(D_1)$ Conjugate Gradient Method in the Fletcher--Reeves form \cite{Fletcher_1964_Function} with the stepsize chosen by one-dimensional minimization. We refer to this algorithm as CGM.
	\item The algorithm proposed in \cite{Gasnikov_2015_Entropy} and based on the idea of Tikhonov's regularization of the dual problem $(D_1)$. In this approach the regularized dual problem is solved by the Fast Gradient Method \cite{Nesterov_2005_Smooth}. We will refer to this algorithm as REG;
	\item Balancing method \cite{Bregman_1967_Proof,Cuturi_2013_Sinkhorn} which is a special type of a fixed-point-iteration method for the system of the optimality conditions for the ROT problem. It is referred below as BAL.
\end{itemize}

The key parameters of the ROT problem in the experiments are as follows

\begin{itemize}
	\item $n := dim(E) = p^2$ -- problem dimension, varies from $2^4$ to $9^4$;
	\item $m_1 := dim(H_1)= 2\sqrt{n}$ and $m_2= dim(H_2) = 0$ -- dimensions of the vectors $b_1$ and $b_2$ respectively;
	\item $c_{ij}$, $i,j = 1,p$ are chosen as squared Euclidean pairwise distance between the points in a $\sqrt{p} \times \sqrt{p}$ grid originated by a 2D image \cite{Cuturi_2013_Sinkhorn,Benamou_2015_Iterative};
	\item $a_1$ and $a_2$ are random vectors in $S_{m_1}(1)$ and $b_1 = (a_1,a_2)^T$;
	\item the regularization parameter $\gamma$ varies from 0.001 to 1;
	\item the desired accuracy of the approximate solution in \eqref{eq:sol_def} is defined by its relative counterpart $\varepsilon_f^{rel}$ and $\varepsilon_g^{rel}$ as follows
	$$ 
	\varepsilon_f = \varepsilon_f^{rel} \cdot f(x(\lambda_0)) \quad \varepsilon_{eq} = \varepsilon_g^{rel} \cdot \|A_1x(\lambda_0) - b_1\|_2, 
	$$
	where $\lambda_0$ is the starting point of the algorithm.
	Note that $\e_{in} = 0$ since no inequality constraints are present in the ROT problem.
\end{itemize}

Figure~\ref{fig:exp_res_cmp_scale} shows the number of iterations for the FGM, BAL and CGM methods depending on the inverse of the regularization parameter $\gamma$. The results for the REG are not plotted since this algorithm required one order of magnitude more iterations than the other methods. In these experiment we chose $n = 2401$ and $\varepsilon_f^{rel} =\varepsilon_g^{rel} = 0.01$. One can see that the complexity of the FGM (i.e. proposed Algorithm \ref{alg:FGM_gen}) depends nearly linearly on the value of $1/\gamma$ and this complexity is smaller than that of the other methods when $\gamma$ is small.

\begin{figure}[h]
\centering
\includegraphics[width=\textwidth]{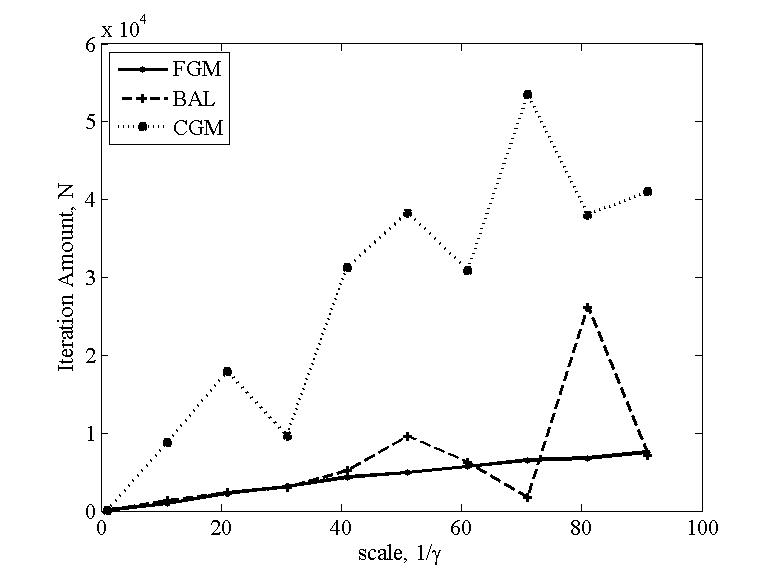}
\caption{Complexity of FGM, BAL and CGM as $\gamma$ varies}
\label{fig:exp_res_cmp_scale}
\end{figure}

Figure~\ref{fig:exp_res_cmp_acc} shows the  the number of iterations for the FGM, BAL and CGM methods depending on the relative error $\varepsilon^{rel}$. The results for the REG are not plotted since this algorithm required one order of magnitude more iterations than the other methods. In these experiment we chose $n = 2401$, $\gamma = 0.1$ and $\varepsilon_f^{rel} =\varepsilon_g^{rel} = \varepsilon^{rel}$. One can see that in half of the cases the FGM (i.e. proposed Algorithm \ref{alg:FGM_gen}) performs better or equally to the other methods.

\begin{figure}[h]
\centering
\includegraphics[width=\textwidth]{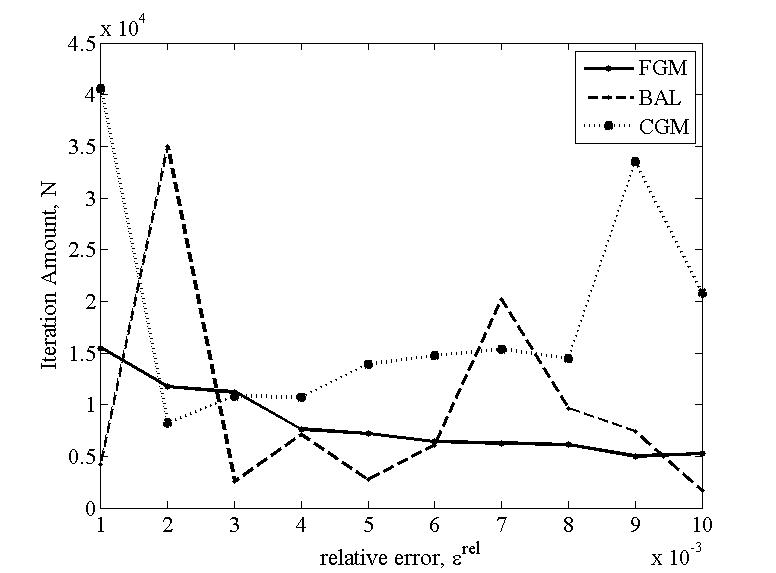}
\caption{Complexity of FGM, BAL and CGM as the desired relative accuracy varies}
\label{fig:exp_res_cmp_acc}
\end{figure}




\section*{Conclusion}
This paper proposes a new primal-dual approach to solve a general class of problems stated as Problem $(P_1)$. Unlike the existing methods, we managed to provide the convergence rate for the proposed algorithm in terms of the primal problem error $|f(\hat{x}_k - Opt[P_1]|$ and the linear constraints infeasibility $\|A_1\hat{x}_k-b_1\|_{2}$, $\|(A_2\hat{x}_k-b_2)_+\|_{2}$. Our numerical experiments show that our algorithm performs better than existing methods for problems of regularized optimal transport which are a special instance of Problem $(P_1)$ for which there exist efficient algorithms.


%
%

\end{document}